\documentclass[a4paper,10pt]{article}

\makeatletter
\def\blfootnote{\gdef\@thefnmark{}\@footnotetext}
\makeatother

\usepackage{soul}

\usepackage{latexsym}
\usepackage{amsmath}
\usepackage{amscd}
\usepackage{amssymb}
\usepackage{authblk}
\usepackage[bbgreekl]{mathbbol}
\usepackage{amsfonts}
\usepackage{stmaryrd}

\usepackage[colorlinks=true, linkcolor=blue, citecolor=blue]{hyperref}
\usepackage{accents}

\usepackage{xcolor}

\usepackage{url}
\usepackage{amsthm}

\def\RCAo{\mathsf{RCA_0}}
\def\WWKLo{\mathsf{WWKL_0}}
\def\WKLo{\mathsf{WKL_0}}
\def\ACAo{\mathsf{ACA_0}}

\def\RCA{\mathsf{RCA_0}}
\def\WWKL{\mathsf{WWKL_0}}
\def\WKL{\mathsf{WKL_0}}
\def\ACA{\mathsf{ACA_0}}

\def\PA{\mathrm{PA}}

\def\E{\exists}
\def\A{\forall}
\def\N{\mathbb{N}}
\def\Z{\mathbb{Z}}
\def\Q{\mathbb{Q}}
\def\R{\mathbb{R}}

\def\rest{{\upharpoonright}}

\newcommand{\mr}[1]{\mathrm{#1}}

\def\P2{\Pi^1_2}

\newcommand{\D}{{\mathbb D}}
\newcommand{\rng}{\mathrm{rng}}

\renewcommand{\labelenumi}{$\arabic{enumi}.$}

\def\SR{\mathrm{SRT}^2_2}

\def\PHt{\mathrm{PH}^2_2}

\def\BII{\mathrm{B}\Sigma^0_2}

\renewcommand{\labelenumi}{$\arabic{enumi}.$}

\newcounter{menum}
{\begin{enumerate}%
\setcounter{enumi}{#1}}%
{\setcounter{menum}{\value{enumi}}\end{enumerate}}

\newtheorem{thm}{Theorem}[section]
\newtheorem{theorem}[thm]{Theorem}

\newtheorem*{theorem*}{Theorem}
\newtheorem{claim}{Claim}[thm]
\newtheorem*{claim*}{Claim}
\newtheorem{prop}[thm]{Proposition}
\newtheorem{lem}[thm]{Lemma}
\newtheorem{cor}[thm]{Corollary}
\newtheorem{proposition}[thm]{Proposition}
\newtheorem{lemma}[thm]{Lemma}

\theoremstyle{definition}

\newtheorem{definition}[thm]{Definition}

\newtheorem{remark}[thm]{Remark}
\newtheorem{question}[thm]{Question}

\def\mathname#1{\ensuremath{\mathsf{#1}}}
\def\RCA{\mathname{RCA}_0}
\def\WKL{\mathname{WKL}_0}
\def\WWKL{\mathname{WWKL}_0}
\def\ACA{\mathname{ACA}_0}

\def\MLR{\mathsf{MLR}}

\def\CR{\mathname{CR}}

\def\mathname#1{\ensuremath{\mathsf{#1}}}

\def\tup{\textup{M}}
\renewcommand\phi{\varphi}
\renewcommand\epsilon{\varepsilon}
\def\restriction{|}

\def\M{\mathcal{M}}
\def\S{\mathcal{S}}

\def\v{\mathbf{v}}

\def\rng{\text{rng}}

\bgroup

\newenvironment{proofclaim}{\begin{trivlist}\item[]{\emph{Proof.}}\rm}{\hfill $\Diamond$\end{trivlist}} 

\DeclareMathAccent{\wtilde}{\mathord}{largesymbols}{"65}
\def\D{\underaccent{\wtilde}{D}}

\begin{document}
\newpage

\title{The reverse mathematics of theorems of  Jordan   and  Lebesgue}

\author[1]{Andr\'e Nies}
\author[2]{Marcus A.\ Triplett}
\author[3]{Keita Yokoyama}

\affil[1]{\small \sf{andre@cs.auckland.ac.nz}}
\affil[2]{\small \sf{marcus.triplett@uq.edu.au}}
\affil[3]{\small \sf{y-keita@jaist.ac.jp}}


 \date{\today}

\blfootnote{Nies' work is partially supported by the Marsden fund of New Zealand. 
Yokoyama's work is partially supported by
 JSPS KAKENHI (grant numbers 19K03601 and 15H03634)
 and JSPS Core-to-Core Program (A.~Advanced Research Networks).
}

\maketitle
\def\Bexp{\mathrm{B}\Sigma_{1}+\mathrm{exp}}
\newcommand\RF{\mathrm{RF}}
\newcommand\tpl{\mathrm{tpl}}
\newcommand\col{\mathrm{col}}
\newcommand\fin{\mathrm{fin}}
\newcommand\Log{\mathrm{Log}}
\newcommand\Ct{\mathrm{Const}}
\newcommand\It{\mathrm{It}}
\renewcommand\PHt{\widetilde{\mathrm{PH}}{}}
\newcommand\BME{\mathrm{BME}_{*}}
\newcommand\HT{\mathrm{HT}}
\newcommand\Fin{\mathrm{Fin}}
\newcommand\FinHT{\mathrm{FinHT}}
\newcommand\wFinHT{\mathrm{wFinHT}}
\newcommand\FS{\mathrm{FS}}
\newcommand\LL{\mathsf{L}}
\newcommand\GPg{\mathrm{GP}}
\newcommand\GP{\mathrm{GP}^{2}_{2}}
\newcommand\FGPg{\mathrm{FGP}}
\newcommand\FGP{\mathrm{FGP}^{2}_{2}}
\newcommand\SGP{\mathrm{SGP}^{2}_{2}}
\newcommand\Con{\mathrm{Con}}
\newcommand\WF{\mathrm{WF}}
\newcommand\bb{\mathbf{b}}

\definecolor{lightred}{rgb}{1,.60,.60}
\newcommand{\andre}[1]{\sethlcolor{lightred}\hl{#1}}
\newcommand{\keita}[1]{\sethlcolor{yellow}\hl{#1}}
\newcommand{\andref}[1]{\sethlcolor{lightred}\hl{\footnote{\hl{#1}}}}
\newcommand{\keitaf}[1]{\sethlcolor{yellow}\hl{\footnote{\hl{#1}}}}

\begin{abstract}
The Jordan decomposition theorem states that every  function $f \colon \, [0,1]  \to \R$ of bounded variation can be written as  the difference of two non-decreasing functions.  Combining this fact with a  result of Lebesgue, every function of bounded variation is  differentiable almost everywhere in the sense of Lebesgue measure. We analyse the strength of   these  theorems in the setting of reverse mathematics.  Over $\RCAo$, a  stronger version  of Jordan's result  where all functions are continuous is equivalent to $\ACA$, while the version  stated is equivalent to $\WKL$.  The result that every function on $[0,1]$ of bounded variation is almost everywhere differentiable is equivalent to $\WWKL$.  To state this equivalence in a meaningful way,  we develop a theory of Martin-L\"of randomness over $\RCAo$. 
\end{abstract}

\tableofcontents

\section{Introduction} A main topic  of reverse mathematics is to determine  the axiomatic  strength of theorems from  classical analysis. For instance, the base system $\RCAo$ proves the intermediate value theorem. Over $\RCAo$,  the fact that every continuous real  function on $[0,1]$  is uniformly continuous is equivalent to the   system $\WKLo$, while  the   Bolzano-Weierstrass theorem (every bounded  sequence of reals has a convergent subsequence) is equivalent to  the stronger system $\ACAo$ (see  \cite[Thms.\  II.6.6, IV.2.3  and  III.3.2]{SOSOA}, respectively). 

Our purpose is to determine the strength of two important, interrelated theorems from analysis.  Interpreting these theorems over $\RCAo$   necessitates to develop some  theory of representations of functions and of Martin-L\"of randomness over this weak base system. 
\subsection{The axiomatic strength of Jordan's decomposition  theorem}   Jordan's  theorem,  dating from 1879,  states that every    function   $f \colon [0,1] \to \R$ of bounded variation can be written as $g-h$ where $g$ and $h$ are nondecreasing functions (see e.g.\  \cite{Carothers:00} for background on real analysis).  One calls the pair $g,h$ a \emph{Jordan decomposition} of $f$. In the setting of real analysis,  the proof that a Jordan decomposition exists is simple: let $g(x)$ be the variation of $f$ from $0$ to $x$, and let $h= g-f$. However,  even if $f$ is computable in the usual sense of computable analysis, the function $g$ is not necessarily computable: the variation of $f$, which equals $g(1)$,  can be any non-negative left-c.e.\ real  by  Rettinger and Zheng~\cite[Thm.\ 3.3.(ii)]{Rettinger.Zheng:04}. They also give in Thm.\ 5.3 an example  of a computable function of bounded variation without any computable Jordan decomposition.  Since the computable sets form a model of $\RCAo$, it follows that Jordan's theorem cannot be proved in $\RCAo$.  

Our first main topic is to determine the strength of Jordan's theorem. It turns out that its  strength depends on which functions we admit in a decomposition. The version where all functions involved are continuous  is equivalent to $\ACAo$. The version where  the non-decreasing functions $g,h$ in the decomposition can be discontinuous is equivalent to $\WKLo$. For the second version,  we need to develop a theory of representing such functions $g, h$  in models of $\RCAo$. In Definition~\ref{df:ratpres} we introduce rational presentations of functions,  which broadly speaking provide  information about all possible inequalities  $g(p)< q$ and $g(p) >q$ for rationals $p,q$, while leaving open equalities.  

Greenberg, Miller and Nies  \cite[Thm.\ 1.4 and Section 2.3]{Greenberg.Miller.ea:nd}, going back to unpublished work with Slaman, built a computable function of bounded variation such that any continuous Jordan decomposition computes  the halting problem, and every   Jordan decomposition allowing discontinuity computes a completion of Peano arithmetic. To prove some  of our  results above,  we adapt  their methods to the setting of reverse mathematics. This will  require considerable additional effort.

\subsection{The axiomatic strength of Lebesgue's theorem on a.e.\ differentiability}

Lebesgue \cite{Lebesgue:1909} proved that every nondecreasing function $f$  is  almost everywhere differentiable.  By Jordan's theorem, it  follows that the same conclusion holds for functions of bounded variation. See e.g.\ \cite[Thm.\ 20.6 and Cor.\ 20.7]{Carothers:00}. 

Our second main topic is the strength of this theorem and of its corollary. We show that with  reasonable interpretations of ``almost everywhere" and ``differentiable" that work over $\RCAo$,   both are  equivalent to weak weak K\"onig's Lemma $\WWKL$ introduced by  Simpson and Yu~\cite{YS1990}, which roughly speaking states that every tree of positive measure has a path. Showing this  requires recasting a fair amount of the methods of Brattka, Miller and Nies~\cite{{BMN2016}} over $\RCAo$. In one important place they used $\Sigma^0_2$-bounding (in the  form of the  infinitary pigeon hole principle), which is not allowed in $\RCAo$. So we have to circumvent this.
To get around the fact that a computable  function of bounded variation may not  have a computable Jordan decomposition, they use a set computing a completion of Peano arithmetic, and   relativize   randomness to it. Since such sets  are unavailable within $\RCAo$, in Lemma~\ref{lemma2} we will instead recast this idea   using an argument  of Simpson and Yokoyama~\cite{SY2011}. They extend  a model of $\WWKL$ to a model of $\WKL$ in a restrictive way, in that  for each of the new sets $A$,  some  set in the given model is random relative to  $A$. This is one of the few examples from earlier years where methods stemming from the algorithmic theory of randomness have been reviewed with the mindset of reverse mathematics. 

 It is interesting that of our two topics,  proving Jordan decomposability requires the stronger systems, even though  differentiation appears to be a more complex  operation than taking a  Jordan decomposition. In fact when we say that $f$ is differentiable at $z$  we cannot assert that the limit of slopes around $z$ exists in the model of $\RCA$, as this would be equivalent to $\ACA$ when considering suitable functions. To get around this we work with the concept of pseudo-differentiability  going back to Demuth \cite{Demuth:75}: $f$ is pseudo-differentiable at $z$ if the slopes get closer and closer to each other as one zooms in on $z$ (similar to the case of  Cauchy sequences).      If $f$ is continuous at $z$ and pseudo-differentiable at $z$, then $f$ is differentiable at $z$ (but  the value of the derivative at $z$   may still not exist in the model).  
 
 We mention that Shafer and the first author~\cite{Nies.Shafer:20}  have recently looked at  further connections between  reverse mathematics and randomness. They consider randomness notions for infinite bit sequences. For instance, they study the reverse mathematical content of a well-known result: the characterization of 2-randomness of a bit sequence $Z$  via the   plain Kolmogorov complexity of initial segments $Z\upharpoonright n$. 
 
\section{Preliminaries}


\subsubsection*{Effectively uniformly continuous functions} We make the following definitions within $\RCA$, borrowing terminology from computable analysis. An \emph{effectively} uniformly continuous function $f:[0,1] \to \mathbb{R}$ is presented by a \emph{Cauchy name}: a sequence $(f_s)_{s \in\mathbb{N}}$ of rational polynomials (or, alternatively,  polygonal functions with rational breakpoints) such that $||f_s - f_r||_\infty \le 2^{-s}$ for all $r > s$. The sequence $(f_s)_{s\in\mathbb{N}}$ is intended to describe $f = \lim_{s \to \infty} f_s$. 
Within $\RCAo$ this definition is equivalent to the definition of continuous functions with a modulus of uniform continuity  given in \cite[Def.\ IV.2.1]{SOSOA}.
Note  that a uniformly continuous function may not have a modulus of uniform continuity within $\RCAo$. In contrast, within $\RCAo$ a continuous function with  a Cauchy name   always has a modulus of uniform continuity,  and vice versa.

\subsubsection*{Functions of bounded variation}
Suppose that  $\Pi = \{t_0, \dots, t_n\}$ is a partition of an interval $[a,b]$, \textit{i.e.}, $a=t_{0}<t_{1},\dots,t_n=b$ (abbreviated by~$\Pi\lhd [a,b]$). We let 
$$V(f, \Pi) = \sum_{i=0}^{n-1} |f(t_{i+1}) - f(t_i)|.$$

\noindent We say that a continuous function $f:[0,1]\to \R$ is of \emph{bounded variation} if there is $k \in \mathbb{N}$ such that $V(f, \Pi) \le k$ for every partition $\Pi$ of $[0,1]$. We define bounded variation in this way in order to avoid  declaring that the supremum exists. We write $\v_{f}(t)=\sup_{\Pi\lhd [0,t]}V(f,\Pi)$, and $\v_{f}=\v_{f}(1)$ in case  the sup exists.
For a given rational number $q\in\Q$, we will use the assertion ``$\v_{f}(t)\le q$'' in the  sense above. It  can be expressed by a $\Pi^{0}_{1}$ formula independent of the sup exists.

%
%
%
%
%

%

\section{Jordan decomposition for effectively uniformly continuous BV functions}\label{section:jdecomp-cont}
Jordan's      theorem states that for every  function $f$ of bounded variation there is  a pair of   non-decreasing functions $g,h$, called a \emph{Jordan decomposition},  such that $f= g-h$. 
For functions $f, g : [0,1] \to \mathbb{R}$, write $$f \le_\mathsf{slope} g \text{ iff } \forall x \forall y [x < y \to f(y) - f(x) \le g(y) - g(x)];$$ i.e., the slopes of $g$ are at least as big as the slopes of $f$. Finding a Jordan decomposition of $f$ is equivalent to finding a non-decreasing function $g$ such that $f \le_\mathsf{slope} g$: If $f= g-h$ for non-decreasing functions $g,h$, then $f \le_\mathsf{slope} g$. Conversely, if $f \le_\mathsf{slope} g$ for a non-decreasing function  $g$,  then $h= g-f$ is nondecreasing and  $f = g - h$.

We consider  a strong  version of the Jordan decomposition theorem: 
the principle $\mathsf{Jordan}_\mathtt{cont}$, which  states that for every  continuous function $f$ of bounded variation,
there exist non-decreasing effectively uniformly continuous functions $g,h:[0,1] \to \mathbb{R}$ such that $f=g-h$. 
Equivalently,
there is a non-decreasing effectively uniformly continuous function $g:[0,1] \to \mathbb{R}$ such that $f \le_{\mathsf{slope}} g$. 

\begin{theorem} \label{jordancont}
The following are equivalent over $\RCAo$.
\begin{enumerate}
 \item $\ACAo$
 \item $\mathsf{Jordan}_\mathtt{cont}$
 \item For every \emph{effectively uniformly} continuous function $f$ of bounded variation,
there exist non-decreasing continuous functions $g,h:[0,1] \to \mathbb{R}$ such that $f=g-h$.
\end{enumerate}
 \end{theorem}
\begin{proof}
To show 1 $\Rightarrow$ 2, given a continuous function $f$ of bounded variation, we construct   a code for a continuous function $\v_f$.
Note that within $\ACAo$, $\v_{f}(t)$ always exists, and one can describe the function $t\mapsto \v_{f}(t)$ by an arithmetical formula.
Thus, one can easily construct a code for $\v_{f}$ by   arithmetical comprehension.
Then $g=\v_{f}$ is the desired function.

\noindent 2 $\Rightarrow$ 3 is trivial.
To show 3 $\Rightarrow$ 1,
 let $$q_n = 1 - 2^{-n - 1} \text{, and } q_{n,s} = q_n - 2^{-n - s - 1}.$$
$\ACA$  is equivalent to the following: if   $h : \mathbb{N} \to \mathbb{N}$ is  an injective function, then the range of $h$ exists \cite[Lemma III.1.3]{SOSOA}. The plan is to encode the  range of $h$   into the variation of an effectively uniformly continuous function~$f$.

For $v \in \mathbb R^+$ and $r \in \mathbb N$, we let $\tup_A(v, r)$ denote a ``sawtooth" function on the interval $A$ with $r$ many teeth of height $v$. Given an injective function  $h$, for each $s \in \mathbb{N}$, define a continuous function $f_s$ as follows. On each interval of the form $I_k = [q_{h(k), k}, q_{h(k), k+1}]$ put 
$$f_s = \begin{cases} \tup_{I_k}(2^{-k}, 2^{k - h(k)}) & \textrm{if $s\ge k > h(k)$,} \\ 0 & \textrm{otherwise.} \end{cases}$$
Let $f_s = 0$ elsewhere. 
Note that on $I_{k}$, $f_{s}=f_{t}$ if $t>s\ge k$ or $k>t>s$, and $\|f_{s}-f_{t}\|_{\infty}\le 2^{-s}$ if $t\ge k>s$.
Thus, the sequence $(f_s)_{s \in \mathbb{N}}$ defines an effectively uniformly continuous function $f = \lim_{s \to \infty}f_s$.
We show that $f$ is of bounded variation with bound $1$. Note that we only need to examine the variation of $f$ on the disjoint intervals $[q_{h(k),k}, q_{h(k),k+1}]$ since $f= 0$ elsewhere. 

Let $m \in \mathbb{N}$. For $k \in \{0, \dots, m\}$, let $\Pi_k$ partition $I_k$. We estimate the variation of $f$ on the interval $\bigcup_{k \le m} I_k$. Without loss of generality we may assume that each partition contains the midpoints and endpoints of the sawteeth defined on $I_k$.\footnote{Indeed, this only refines the partition and provides an improved estimate.} This allows us to easily compute the variation of $f$ as the piece-wise combination of non-decreasing functions. 
For all $s \ge m$ one has
\begin{align*}
\sum \limits_{k = 0}^m V(f, \Pi_k) = \sum \limits_{k = 0}^m V(f_s, \Pi_k) \le \sum \limits_{k = 0}^m 2^{-h(k) + 1}  < 1,
\end{align*}
which establishes the desired bound.

By $\mathsf{Jordan}_\mathtt{cont}$, take $g:[0,1] \to \mathbb{R}$ non-decreasing and continuous such that $f \le_\mathsf{slope} g$.
Given that the range of $h$ is encoded in the variation of $f$, we will use the (easily computable) variation of $g$ on the interval $[q_{n, k}, q_{n, k+1}]$ to bound to possible pre-images   of $n$ under $h$.  

Define a $\Delta^{0}_{1}$ definable function $\gamma :\mathbb{N} \to \mathbb{N}$ such that $g(q_n) - g(q_{n, \gamma(n)}) < 2^{-n}$ as follows. There is a $\Sigma^0_0$ formula $\theta(n, m, k)$ such that $$\exists m \,  \theta(n, m, k) \leftrightarrow g(q_n) - g(q_{n,k}) < 2^{-n}.$$
Since $g$ is continuous and $\lim_{s \to \infty} q_{n,s} = q_n$ one has $\forall n \exists k \exists m \, \theta(n, m, k).$ As this sentence is $\Pi^0_2$, given any instance of the variable $n$ one can effectively obtain a witness $\langle m, k \rangle$ for the $\Sigma^0_1$ formula $\exists k \exists m \theta(n, m, k)$ (see, e.g.\ \cite[Theorem II.3.5]{SOSOA}). Thus 
we may put $\gamma(n) = k$, where $\langle m, k \rangle$ is least such that $\theta(n,m,k)$ holds.
 
Now if $h(k) = n<k$ then by the monotonicity of $g$, $$g(q_n) - g(q_{n,k}) \ge g(q_{n,k+1}) - g(q_{n, k}).$$ Let $\Pi$ be a partition of $[q_{n,k}, q_{n,k+1}]$ containing the endpoints and midpoints of each sawtooth defined on that interval. Then since $g - f \le g$ and the variation of an increasing function is the difference of its values at its endpoints one has \begin{align*} 2^{-n + 1} = V(f, \Pi) & = V(g - (g - f), \Pi) \le V(g, \Pi) + V(g - f, \Pi) \le 2(g(q_{n,k+1}) - g(q_{n,k})).\end{align*}
Thus $g(q_n) - g(q_{n,k}) \ge 2^{-n}$, and then $k < \gamma(n)$. Hence 
$$n \in \rng(h) \leftrightarrow \exists k < \max\{\gamma(n),n+1\} [h(k) = n], $$
 so the range of $h$ exists by $\Delta^0_1$ comprehension.
\end{proof}

\section{Jordan decomposition for BV functions of rational domain} \label{section:jdecomp-rat}

 In the foregoing section,  we required that  a Jordan decomposition  consist of effectively uniformly continuous functions. Then   the Jordan decomposition theorem  has the same axiomatic strength as  $\ACA$.
 To see this, we encoded the range of an injective function $h$ into the variation of a function~$f$ of bounded variation. A Jordan decomposition of  $f$ into uniformly continuous functions allowed us to recover enough information to decide whether some number was the image of another under the injective function $h$.
   
We now relax the requirement on the Jordan decomposition by only stipulating that the decomposition is given by functions which are defined on the rationals. 
 Such functions can be represented by finite strings that cumulatively describe the behaviour of the function at each rational.  We will see that   such simple  objects do not allow the  encoding of  sets of high complexity.

Greenberg, Miller and Nies  \cite{Greenberg.Miller.ea:nd}   proved that there is a  computable function $f$ on $[0,1]$ of bounded variation  such that every Jordan decomposition of $f$  in this weak sense is PA-complete.
One direction of our argument, 4 $\Rightarrow$ 1 of Theorem~\ref{jordanq}, is based on their proof; extra effort is required to make it work over  $\RCA$
as a base theory.

\subsection{Rational presentations of functions}
Let $[0,1]_\mathbb{Q} := [0,1]\cap \mathbb{Q}$. We present a function $g : [0,1]_\mathbb{Q} \to \mathbb{R}$ by a set $Z\subseteq[0,1]_{\Q}\times\Q$ in the following way.
We require that  $(p,q)\in Z$ if $g(p) < q$, and  $(p,q)\notin Z$ if $g(p) > q$.
We leave open whether $(p,q)\in Z$   in case that  $g(p)=q$. 
The formal definition follows. 
\begin{definition} \label{df:ratpres}A set $Z\subseteq[0,1]_{\Q}\times\Q$ is called  a \emph{rational presentation}   if
\begin{itemize}
 \item[(i)]    for any $p\in[0,1]_{\Q}$, there exist $q,q'\in\Q$ such that $(p,q)\in Z$ and $(p,q')\notin Z$, and
 \item[(ii)]  for any $p\in[0,1]_{\Q}$ and for any $q,q'\in \Q$ with $q < q'$,  $(p,q)\in Z$ implies $(p,q')\in Z$.
 \end{itemize}
A rational presentation  $Z$ determines  a  function $g_{Z}: [0,1]_{\Q} \to \mathbb{R}$ via  \[g_{Z}(p) = \inf \{q \in \mathbb{Q} : (p,q)\in Z\}.\]  
We say that $Z$ is a rational presentation of $g_Z$ (and also of any function on $[0,1]$ extending $g_Z$). \end{definition}
One can determine    $g_{Z}(p)$ within $\RCAo$ since for any $n\in\N$, one can effectively find $q,q'\in \Q$ such that $(p,q)\in Z$, $(p,q')\notin Z$,  and $|q-q'|\le 2^{-n}$.
Even though a rational presentation of a function  is not unique if the function has some  rational value, we  sometimes  identify $Z$ with $g_{Z}$.  

For given $x,y,z\in \Q$ and a rationally presented function $g_{Z}:[0,1]_{\Q}\to\R$,
the assertion ``$g_{Z}(x)-g_{Z}(y)\le b$'' is expressed  by a  $\Pi^{0}_{1}$ formula with free variables $Z, x,y,b$: 
%
\[\forall p_{0},p_{1},q_{0},q_{1}\in\Q  \, [p_{0}=x  \wedge p_{1}=y   \wedge (p_{0},q_{0})\notin Z\wedge (p_{1},q_{1})\in Z\ \to  \ q_{0}-q_{1}\le b]. \]

Similarly, ``$g_{Z}(x)-g_{Z}(y)\ge b$'' can be expressed by a $\Pi^{0}_{1}$ formula, and ``$g_{Z}(x)-g_{Z}(y)< z$'' and ``$g_{Z}(x)-g_{Z}(y)> z$''   by $\Sigma^{0}_{1}$ formulas.
Thus, the assertion ``$\v_{g_{Z}}(x)\le z$'' is also expressed by a $\Pi^{0}_{1}$ formula.
(Here, we only consider partitions with rational end points.)
We say that $g_{Z}$ is  of bounded variation if $\v_{g_{Z}}(1)\le k$ for some $k\in\N$.



A function $f:[0,1]_{\Q}\to \R$ can be canonically  encoded by a function $f:[0,1]_{\Q}\times\N\to \Q$ such that $|f(p,n)-f(p,n+k)|\le 2^{-n}$ for each $p\in [0,1]_{\Q}$ and $n,k\in\N$. Rational presentations are essentially sufficient for  presenting all real-valued  functions   on $[0,1]_{\Q}$: as we show next,
within $\RCAo$ any function $f:[0,1]_{\Q}\to \R$ has a rational presentation up to a   vertical shift.

%
\begin{lem}[$\RCAo$]\label{lem:shift-for-r-presentation}
For every  function $f:[0,1]_{\Q}\to \R$, there exists a real $a\in \R$ such that $f(p)+a \not \in \Q$  for any $p\in [0,1]_{\Q}$.
\end{lem}
\begin{proof}
Let $\{p_{i}\}_{i\in\N}$ be an enumeration of $[0,1]_{\Q}$.
We   recursively define a sequence of rationals $\{a_{i}\}_{i\in\N}$ as follows.
Let $a_{0}=0$.

For given $a_{i}\in \Q$ we let  $a_{i+1}\in\Q$ such that  $|a_{i}-a_{i+1}|<4^{-i}$  and  $|f(p_{i})-a_{i+1}|>4^{-i}/2 $. 
One can always pick such $a_{i+1}$ effectively since the required condition on $a_{i+1}$ given $a_i$ is $\Sigma^{0}_{1}$.
(Here we use the well-known fact  that a dependent choice function for a $\Sigma^{0}_{1}$ binary predicate of  numbers is available within $\RCAo$.
See, e.g., the argument in the proof of \cite[Theorem II.5.8]{SOSOA}, or \cite[Theorem~2.1]{Yokoyama-thesis}.)
Put $a=\lim_{n\to\infty}a_{i}$ and note that $|a - a_{i+1}| \le 4^{-i}/3$. Therefore $|f(p_{i})-a| > 4^{-i}/2 - 4^{-i}/3>0$. 
\end{proof}
\begin{prop}\label{prop:rational-presentation-in-RCA} \mbox{}
\begin{enumerate}
 \item[(i)] $\WKLo$ proves that every function $f:[0,1]_{\Q}\to \R$ has a rational presentation.
 
  \item[(ii)] $\RCAo$ proves that every function $f:[0,1]_{\Q}\to \R$ has a rational presentation up to a vertical shift.  That is, there exists a rational presentation $Z$ and a real $r\in\R$ such that $f+r=g_{Z}$.
\end{enumerate}
\end{prop}
\begin{proof}
(i) Consider the $\Sigma^{0}_{1}$-definable
 sets $\mathcal{A}=\{(p,q)\in [0,1]_{\Q}\times \Q: f(p)< q\}$ and $\mathcal{B}=\{(p,q)\in [0,1]_{\Q}\times \Q: f(p)> q\}$.
To obtain a rational presentation for $f$, it  suffices  to find a set $Z\subseteq [0,1]_{\Q}\times \Q$ such that $\mathcal{A}\subseteq Z\subseteq [0,1]_{\Q}\times \Q\setminus \mathcal{B}$.
Such a set $Z$ is obtained by an instance of $\Sigma^{0}_{1}$-separation, an axiom scheme  which follows from  $\WKLo$ over $\RCAo$ (\cite[Lemma~IV.4.4]{SOSOA}).

\noindent (ii)  We may assume that $f$ avoids rational numbers after passing to a vertical shift of $f$  via  Lemma~\ref{lem:shift-for-r-presentation}. Then  both   $\mathcal{A}$ and $\mathcal{B}$ are $\Delta^{0}_{1}$. 
\end{proof}
\begin{cor}\label{cor:rational-presentation-in-RCA}
\begin{enumerate}
 \item[(i)] $\WKLo$ proves that the restriction to $[0,1]_{\Q}$ of any continuous function $f:[0,1]\to \R$ has a rational presentation.
  \item[(ii)] $\RCAo$ proves that  the restriction to $[0,1]_{\Q}$ of any  continuous function $f:[0,1]\to \R$ has a rational presentation on $[0,1]_{\Q}$ up to a vertical shift.
\end{enumerate}
\end{cor}
\begin{proof}
Within $\RCAo$, one can effectively find the value of a continuous function.
Thus, the restriction to $[0,1]_{\Q}$ of a continuous function $f:[0,1]\to\R$ has a canonical encoding.
\end{proof}

Taking a vertical shift is essential in the above discussion: an effectively uniformly continuous function itself might not have a rational presentation within $\RCAo$. To see this apply the next fact to a recursively inseparable pair. 

\begin{prop}\label{prop:con-vs-rp}
Given a disjoint pair  $A,B$ of c.e.\ sets,  there is a computable nondecreasing function $f \colon [0,1] \to \R$ such that every  rational presentation computes a set $X$ such that $A \subseteq X \subseteq \N \setminus B$.  
\end{prop}
\begin{proof}[Sketch of proof] 
We define a uniformly computable  sequence of reals $(r_e) $ such that $r_e$ is very close to $2^{-e}$; say  $|r_e - 2^{-e}| \le 2^{-2e}$. The function $f$ is then obtained by linear interpolation between the values $f(2^{-e})= r_e$; in particular $f(0)=0$ and $f$ is computable in the usual sense of computable analysis.

We  define $r_e$ using a Cauchy name, as follows. Initially we let  $r_e = 2^{-e}$. 
If  stage $s$ is least such  that $s\ge 2e$ and $e \in A_s$   we subtract $2^{-s}$ to $r_e$ and leave $r_e$  at this value. 
If  stage $s$ is least such  that $s\ge 2e$ and  $e \in B_s$ we add $2^{-s}$ from  $r_e$ and leave $r_e$  at this value.  

If $Z$ is  a rational presentation of $f$, let $X = \{e \colon \, (2^{-e},2^{-e}) \in Z\}$. If $e \in A$ then $f(2^{-e})< 2^{-e}$ and hence $e \in X$. If $e \in B$ then $f(2^{-e}) >  2^{-e}$ and hence $e \not \in X$. \end{proof}

%
%
%

\begin{prop}\label{prop:con-vs-rp-inRM}
The assertion ``every effectively uniformly continuous function $f:[0,1]\to \R$ has a rational presentation'' implies $\WKLo$ over $\RCAo$.
\end{prop}
\begin{proof}  It is routine to transfer the computability theoretic proof above into an argument that the given assertion implies  $\Sigma^0_1$ separation over $\RCAo$. By \cite[Lemma~IV.4.4]{SOSOA} the scheme of $\Sigma^0_1$ separation is equivalent to $\WKL$ over $\RCAo$.
\end{proof}
The above discussions recast the problem of converting Cauchy representations and Dedekind cuts for reals studied by Hirst \cite{Hirst2007}. Proposition~\ref{prop:con-vs-rp-inRM} can be viewed as a strengthening of \cite[Theorem~7]{Hirst2007}.

\subsection{Jordan decomposition by rationally presented functions}

We modify the $\le_\mathsf{slope}$ notation for functions of rational domain. For $f, g : \subseteq [0,1] \to \mathbb{R}$ we let $$f \le^*_\mathsf{slope} g \text{ iff } \forall x, y \in [0,1]_\mathbb{Q} [x < y \to (f(y) - f(x) \le g(y) - g(x))].$$
We will use the following ``folklore"  fact for the next theorem.
\begin{lem}[$\WKLo$]\label{stWKL}
Every $\Pi^{0}_{1}$ definable infinite tree $T\subseteq 2^{<\N}$ has a path.
\end{lem}
\begin{proof}
Suppose  that $\tau\in T\leftrightarrow \A n \, \theta(n,\tau)$.
By $\Delta^{0}_{1}$ comprehension, there exists a tree 
\[\bar{T}=\{\tau: \A n\le |\tau|\A \sigma\preceq\tau \, \theta(n,\sigma)\}.\]
Then, $T\subseteq \bar{T}$, and any path of $\bar{T}$ is a path of $T$ by the definition of $T$.
By $\WKLo$, $\bar{T}$ has a path, thus $T$ has a path.
\end{proof}
\begin{thm}[$\WKLo$]\label{jordanqq}
For every rationally presented function $f : [0,1]_\mathbb{Q} \to \mathbb{R}$ of bounded variation, there is a rationally presented non-decreasing function $g : [0,1]_\mathbb{Q} \to \mathbb{R}$ such that $f \le^*_\mathsf{slope} g$.
\end{thm}
\begin{proof}
Let   $M\in\N$ such that $\v_{f}\le M$.
We fix an effective listing $( p_n, q_n )_{n \in \mathbb{N}}$ of all elements of $[0,1]_{\Q}\times\Q$,
and identify $Z:\N\to\{0,1\}$ as $\{(p_{n},q_{n}):Z(n)=1\}\subseteq [0,1]_{\Q}\times\Q$.
We construct a binary tree $T$ such that any path $Z$ through $T$ encodes a non-decreasing function $g :[0,1]_\mathbb{Q} \to \mathbb{R}$ with $f \le^*_{\mathsf{slope}} g$. To do so, we   ensure that  the following   conditions hold:
\\

\begin{tabular}{l l}
$\mathcal{R}_0:$ & for any $r\in\N$, $0\le g(p_r) \le M$; \\
$\mathcal{R}_1:$ & for any $r,s\in\N$, if $p_s \le p_r$, $q_s \ge q_r$, and $g(p_s) > q_s$ then $g(p_r) > q_r$; \\
$\mathcal{R}_2:$ & for any $r,s\in\N$, if $p_s \le p_r$ then $f(p_r) - f(p_s) \le g(p_r) - g(p_s)$.
\end{tabular}\
\\

\noindent Here, $\mathcal{R}_1$ guarantees that any $g$ encoded by a path $Z_g$ through $T$ is non-decreasing, and $\mathcal{R}_2$ guarantees the slope condition.
Formally, 
we will consider a $\Pi^{0}_{1}$ definable tree $T$ to be the set of all $\tau \in 2^{<\mathbb{N}}$ such that 
\begin{align*}
\text{(r0)}& \ \forall r < |\tau|  \Big[(q_{r}\le 0\to \tau(r) = 0) \wedge (q_{r}\ge M\to \tau(r) = 1)\Big],  \\
\text{(r1)}& \ \forall r, s < |\tau|  \Big[(p_r \le p_s  \land   q_r  \ge  q_s \land \tau(r) = 0) \to \tau(s) = 0\Big],  \\
\text{(r2)}& \ \forall r, s < |\tau| \Big[ (p_r \le p_s \land \tau(r) = 0 \land  \tau (s) = 1) \to | f(p_s) - f(p_r) | \le q_{s}-q_{r}\Big].
\end{align*}

To see that $T$ is infinite, notice that since $f$ is of bounded variation, the string $Z_{v_{f}}\restriction_k$ defined as $s\in Z_{v_{f}}\restriction_k$ iff $\v_{f}(p_{s})<q_{s}$ and $s<k$, which is available from bounded $\Sigma^{0}_{1}$ comprehension (see \cite[Theorem~II.3.10]{SOSOA}), is an element of $T$ for every $k \in \mathbb{N}$. Thus by Lemma~\ref{stWKL}, $T$ has a path $Z$.
By (r0) and (r1) (for the case $p_{r}=p_{s}$), $Z$ encodes a rational presentation.
Let $g$ be the unique function  $ [0,1]_\mathbb{Q} \to \mathbb{R}$ defined as $g=g_{Z}$.

\begin{claim} The function $g$ is non-decreasing. \end{claim}

\begin{proofclaim} Take $x, y \in [0,1]_\mathbb{Q}$ with $x < y$. Let $q \in \mathbb{Q}$. It suffices to show that if $g(x) > q$ then $g(y) > q$. There are $r, s \in \mathbb{N}$ such that $p_s = y$, $q_s = q$, $p_r = x$, and $q_r = q$. If $g(p_r) > q_r$ then $Z(r)=0$, and then by clause (r1), $Z(s)=0$, which means $g(p_s) > q_s$. \end{proofclaim}

\begin{claim}$f \le^*_\mathsf{slope} g$. \end{claim}

\begin{proofclaim}
Let $x, y \in [0,1]_\mathbb{Q}$ such that $x < y$.
It is enough to show that for any $q\in\Q$ such that $g(y)-g(x)<q$, $|f(y)-f(x)|<q$.
By the definition of $g$, one can find $r,s\in\N$ such that $x=p_{r}$, $y=p_{s}$, $g(p_{r})>q_{r}$, $g(p_{s})<q_{s}$ and $q_{s}-q_{r}<q$.
Then, by (r2) we have $|f(y)-f(x)|=|f(p_{s})-f(p_{r})|\le q_{s}-q_{r}<q$.
\end{proofclaim}
Thus, this $g=g_{Z}$ is the desired function.
\end{proof}


 It is a  well-known fact that every $\Pi^{0}_{1}$-class with only finitely many members has  a computable member.  
 Greenberg,  Nies and Slaman  used this fact   to build a  computable function $f$ on $[0,1]$ of bounded variation such that any  Jordan decomposition of $f$ is  PA-complete; see \cite[Section 2.3 and in particular Prop.\ 2.9]{Greenberg.Miller.ea:nd}.
  A natural formalization within $\RCAo$ of this fact is as follows: 
if an infinite tree $T$ has only boundedly many incomparable nodes that  are extendible, then $T$ has a  path that is computable relative to $T$. Simpson and Yokoyama~\cite{SY201X} showed that this formalization
 already requires $\Sigma^{0}_{2}$-induction.
Instead, we will use the following lemma.
\begin{lem}[$\RCAo$]\label{p-bound-tree}
Let  $T\subseteq 2^{<\N}$ be an infinite tree. If there is a bound on the cardinality of an arbitrary prefix-free subset of $T$, then $T$ has a path.
\end{lem}
\begin{proof}
Take $K\in\N$ so that $|P|<K$ for any prefix-free set $P\subseteq T$.
By $\Sigma^0_1$ induction, take
\begin{align}
 \label{line1} k = \max\{i\le K : \text{there is a prefix-free set $P \subseteq T$ with $|P| = i$}\}.
 \end{align}
 Let $P_k \subseteq T$ witness (\ref{line1}).
 Let $\sigma = \max P_k$, where the max is taken with respect to the usual integer encoding of binary strings.
  Let $\ell = \max\{|\tau| : \tau \in P_k\}$.
  Any $\tau \in T$ with $|\tau| > \ell$ must extend an element of $P_k$, and can have at most one successor.
By the pigeonhole principle (which is available from $\Sigma^{0}_{1}$ induction), there exists $\tau\in P_{k}$ with  infinitely many extensions in $T$.
 Since each extension of $\tau$ of  length exceeding $\ell$ has exactly one successor, we can effectively find a path through $ T$ extending $\tau$.
 \end{proof}

We now establish  the main theorem of this section.
\begin{theorem} \label{jordanq}
The following are equivalent over $\RCAo$.
\begin{enumerate}
 \item $\WKLo$.
 \item $\mathsf{Jordan}_\mathbb{Q}$: for every rationally presented function $f$ of bounded variation, there is a rationally presented non-decreasing function $g : [0,1]_\mathbb{Q} \to \mathbb{R}$ such that $f \le^*_\mathsf{slope} g$.
 \item For every \emph{continuous} function $f$ of bounded variation, there is a rationally presented non-decreasing function $g : [0,1]_\mathbb{Q} \to \mathbb{R}$ such that $f \le^*_\mathsf{slope} g$.
 \item For every effectively uniformly continuous function $f$ of bounded variation which has a rational presentation, there is a rationally presented non-decreasing function $g : [0,1]_\mathbb{Q} \to \mathbb{R}$ such that $f \le^*_\mathsf{slope} g$.
\end{enumerate}
 \end{theorem}

\begin{proof}
1 $\Rightarrow$ 2 is   Theorem~\ref{jordanqq},
2 $\Rightarrow$ 3 is a direct consequence of Corollary~\ref{cor:rational-presentation-in-RCA}(ii),
and 3 $\Rightarrow$ 4 is trivial.
We show 4 $\Rightarrow$ 1. We reason within $\RCA$.
Let $T \subseteq 2^{<\mathbb{N}}$ be an infinite binary tree. 
Assume for a contradiction  that $T$ has no path.
Let $\widetilde{T} = \{\tau \in 2^{<\mathbb{N}} : \tau \not \in T \land \tau \restriction_{(|\tau| - 1)} \in T\}$. Without loss of generality we may assume that $\widetilde{T}$ is infinite.
Consider the  $\Sigma^{0}_{1}$ definable set 
\[\mathrm{Nonext}(T):=\{\tau \in T : \tau \text{ has  only finitely many extensions in } T\}.\]
Then, by \cite[Lemma~II.3.7]{SOSOA}, there exists a one-to-one function $h:\N\to\N$ such that $\rng(h)=\mathrm{Nonext}(T)$.
(Here, we identify a binary string with its usual integer encoding.)

Let $(\widetilde{\sigma}_k)_{k\in\mathbb{N}}$ be an enumeration of $\widetilde{T}$ such that $|\widetilde{\sigma}_i| \le |\widetilde{\sigma}_{i+1}|$. Note that for any $k, \ell \in \mathbb{N}$,
 \begin{align} \label{line2}
  |\widetilde \sigma_k | \le \ell \rightarrow k \le 2^\ell.
 \end{align}

For all $\sigma \in 2^{<\mathbb{N}}$ put $I_\sigma = [0.\sigma, 0.\sigma + 2^{-|\sigma|}]$. For each $s \in \mathbb{N}$ define a polygonal function $f_s :[0,1] \to \mathbb{R}$ as follows.
 On the interval $I_{\widetilde{\sigma}_k}$ set 
\begin{align}\label{line3}
f_s = \begin{cases} \tup_{I_{\widetilde{\sigma}_k}}(2^{-k}, 2^{k - h(k)})  & \textrm{if $s\ge k> h(k)$,} \\ 0 & \textrm{otherwise.} \end{cases} 
\end{align}
Let $f_s = 0$ elsewhere. Then $(f_s)_{s\in\mathbb{N}}$ defines an effectively uniformly continuous function $f = \lim_{s}f_s$.
By Corollary~\ref{cor:rational-presentation-in-RCA}(ii), one may replace $f$ with a vertical shift and then  assume that $f$ has a rational presentation.

We show that $f$ is of bounded variation. As in the proof of Theorem~\ref{jordancont},
we only need to consider the variation of $f$ on the disjoint intervals $I_{\widetilde{\sigma}_k}$. Let $m \in \mathbb{N}$, and for each $k \le m$ let $\Pi_k$ be a partition of $I_{\widetilde{\sigma}_k}$ containing the midpoints and endpoints of each sawtooth defined on that interval. For all $s \ge m$ one has
\begin{align*}
\sum \limits_{k = 0}^m V(f, \Pi_k) = \sum \limits_{k = 0}^m V(f_s, \Pi_k) \le  \sum \limits_{k = 0}^m 2^{-h(k) + 1}  < 1,
\end{align*}
as required.

By our hypothesis in (5.), there exists a rationally presented non-decreasing function $g :[0,1]_\mathbb{Q} \to \mathbb{R}$
such that $f \le^*_{\mathsf{slope}} g$. Let   $\Delta : \mathbb{N} \to \mathbb{R}$ be the function given  by $\Delta(k) = \max\{g(0.\sigma + 2^{-|\sigma|}) - g(0.\sigma) : \sigma \in T \land |\sigma| = k\}.$ Note that  $\Delta$ is non-increasing. 

There are two cases to consider. If $\lim_{n \to \infty} \Delta(n) = 0$, then $g$ behaves like a continuous function. This provides a decomposition of $f$ that allows us to use an argument similar to the one in Theorem \ref{jordancont} to prove the existence of $\rng(h)$. One can then find a path through $T$ by avoiding this set.

Otherwise, there is a jump-type discontinuity of $g$. The intervals around this point correspond to strings which form an infinite subtree $\widehat T$ of $T$. One can bound the size of any prefix-free subset of $\widehat T$ using the size of this jump, and thus effectively find a path through $\widehat T$.

We now analyse the two cases in detail.

\noindent\textbf{Case 1. } $\lim_{n\to\infty}\Delta(n) = 0$.
Take $\gamma : \mathbb{N} \to \mathbb{N}$ such that $\Delta(\gamma(n)) < 2^{-n}$.
Such $\gamma$ exists by $\Delta^{0}_{1}$ comprehension since ``$\Delta(k)<2^{-n}$'' can be described by a $\Sigma^{0}_{1}$ formula.
 If $h(k) = n<k$ then by (\ref{line3}),
 $$g(0.\widetilde{\sigma}_k + 2^{-|\widetilde{\sigma}_k|}) - g(0.\widetilde{\sigma}_k) \ge V(\tup_{I_{\widetilde{\sigma}_k}}(2^{-k}, 2^{k - h(k)}), \Pi_k)= 2^{-h(k)+1}\ge  2^{-n}.$$
  Hence $|\widetilde{\sigma}_k| \le \gamma(n)$, and then by (\ref{line2}) $k \le 2^{\gamma(n)}$. This gives
   $$n \in \rng(h) \leftrightarrow \exists k \le \max\{2^{\gamma(n)},n\}[h(k) = n],$$
    so $\rng(h)=\mathrm{Nonext}(T)$ exists by $\Delta^0_1$ comprehension. 
Thus $\mathrm{Ext}(T)=T \setminus \mathrm{Nonext}(T)$ exists.
Now, one can construct a path of $T$ by an easy primitive recursion: starting from the empty string, one can choose the left most immediate extension of a given string in $\mathrm{Ext}(T)$.
\\

\noindent\textbf{Case 2. } There exists $M \in \mathbb{N}$ such that $\forall m \exists n [n > m \land \Delta(n) > 2^{-M}]$
Then there are infinitely many strings $\sigma\in T$ such that
$$g(0.\sigma + 2^{-|\sigma|}) - g(0.\sigma) >2^{-M}.$$ 
Without loss of generality, we may take   $K\in\N$ such  that $0\le g(0)<g(1)\le K$.
Let $L=\{i/2^{M+2}: 0\le i\le K2^{M+2}\}$.
Given a   rational presentation  $Z$ of $g$, for  $x,y,z\in \Q$, we write $g(y)-g(x)\ge_{L}z$ if there exist $r,s\in\N$ such that $x=p_{r}$, $y=p_{s}$, $q_{r},q_{s}\in L$, $Z(r)= {1}$ (which implies that  $g(p_{r})<q_{r}$), $Z(s)={0}$ (which implies that  $g(p_{s})>q_{s}$) and $q_{s}-q_{r}\ge z$.
Since $L$ is finite, $g(y)-g(x)\ge_{L}z$ is actually $\Delta^{0}_{1}$ statement as  we only need to check finitely many $r$ and $s$.
Note   that if $g(y)-g(x)>2^{-M}$, then $g(y)-g(x)\ge_{L}2^{-M-2}$ since there are at least two points in $L\cap(g(x),g(y))$.

The tree    $\widehat T\subseteq T$  defined by  $\widehat T=\{\sigma\in T: g(0.\sigma + 2^{-|\sigma|}) - g(0.\sigma) \ge_{L}2^{-M-2}\}$.
  is an infinite subtree of $T$.
(The case assumption guarantees that $\widehat T$ is infinite, and  $\widehat T$ is closed under prefixes because $g$ is non-decreasing.)
We verify  the cardinality of any prefix-free subset of $\widehat T$ is bounded. 
For any prefix-free $P \subseteq \widehat T$, we have
 $$|P|2^{-M-2} \le  \sum \limits_{\sigma \in P} g(0.\sigma + 2^{-|\sigma|}) - g(0.\sigma) \le g(1) - g(0) \le K.$$ 
Thus, $|P|\le K2^{M+2}$.
Hence by Lemma~\ref{p-bound-tree}, $\widehat T$ has a path, and thus $T$ has a path.
\end{proof}

We thank Paul Shafer for providing helpful comments on a previous version of this proof.

%
%

\section{Martin-L\"of randomness  within $\RCAo$} To study Martin-L\"of random reals within $\RCAo$, we need to define  a notion of uniform measure for open sets.
A  set of binary strings $S\subseteq 2^{<\N}$ is a \emph{code} of an  open set $U\subseteq 2^{\N}$ if  $Z\in U\leftrightarrow \E \sigma\in S \, [Z\succ \sigma]$. We write $[[S]]$ for the open set coded by $S$.
(Note that a $\Sigma^{0}_{1}$-definable set may be used to code an open set, but one can easily find an (existing) set which codes the same open set  by $\Delta^{0}_{1}$-comprehension.)
Given such a code  $S\subseteq 2^{<\N}$,   let $T_{S}:=\{\sigma\in 2^{<\N}: \A n<|\sigma|(\sigma\rest n\notin S)\}$. Note that $T_{S}$ forms a tree, which we view as a  code of  the complement of $U$.
We first define the measure for a code $S$ of an open set,   and also of its complementary code $T=T_{S}$:
\[\mu(S):=1-\mu(T)=1-\lim_{n\to\infty}\frac{|\{\sigma\in T: |\sigma=n\}|}{2^{n}}.\]
Note that if $S$ is prefix free, then $\mu(S)=\sum_{\sigma\in S}2^{-|\sigma|}$.
The existence of the limit is not  guaranteed within $\RCAo$, but one can still make  assertions such as $\mu(S)\le a$ or $\mu(T_{S})\ge a$, which can be expressed  by  $\Pi^{0}_{1}$-formulas.  


  $Z\in 2^{\N}$ is said to be Martin-L\"of random relative to $X$ if for any $X$-computable sequence of codes for open sets $\langle S_{n}: n\in\N \rangle$ such that $\mu(S_{n})\le 2^{-n}$, there exists $n\in\N$ such that $Z\notin [[S_{n}]]$.
The assertion ``for any $X$ there exists a Martin-L\"of random real relative to $X$'' is equivalent to $\mr{WWKL}$  by Simpson and Yu \cite{YS1990}.
We always identify a real $z\in [0,1] $ that is not a dyadic rational with its unique  binary expansion viewed as an element of $2^{\N}$.

Besides the fact that the measure of an open set may not exist as a real in the model, there is another problem when developing measure theory within $\RCAo$.
There might exist two codes for open sets $S_{1}$ and $S_{2}$ such that $\A x\in 2^{\N}(x\in [[S_{1}]]\leftrightarrow x\in [[S_{2}]])$ but $\mu(S_{1})\neq \mu(S_{2})$. Thus the value of $\mu$   depends on codes.
We  define the measure for an open set $U\subseteq 2^{\N}$ as
\[\bar{\mu}(U):=\sup\{\mu(S)\mid U=[[S]]\}.\]
This definition   agrees with the   internal measure of  open sets defined  in  \cite[p.\ 174]{YS1990}. 
Within $\WWKLo$, $[[S_{1}]]=[[S_{2}]]$ implies $\mu(S_{1})=\mu(S_{2})$, thus $\mu$ and $\bar \mu$ coincide.
Fortunately, the definition of Martin-L\"of randomness will not be affected even if the two don't coincide. We take any of the two equivalent conditions  below as a definition in the context of $\RCAo$ that $Z$ is not Martin-L\"of random relative to $X$. 
\begin{prop}[$\RCAo$]\label{prop:ML-random-in-RCA}
 The following are equivalent  for $Z,X\in 2^{\N}$.
\begin{enumerate}
 \item There exists an $X$-computable sequence  $\langle S_{i}\mid i\in\N \rangle$  of codes of open sets such that $\mu(S_{i})\le 2^{-i}$ and $Z\in \bigcap_{i\in\N}[[S_{i}]]$.
 \item There exists an $X$-computable sequence $\langle S_{i}\mid i\in\N \rangle$  of codes of open sets  such that $\bar \mu([[S_{i}]])\le 2^{-i}$ and $Z\in \bigcap_{i\in\N}[[S_{i}]]$.
\end{enumerate}
\end{prop}
\begin{proof}
2 $\Rightarrow$ 1 is trivial.
To show 1 $\Rightarrow$ 2, let $\langle S_{i}\mid i\in\N \rangle$ be an $X$-computable sequence of  codes  for open sets such that $\mu(S_{i})\le 2^{-i}$ and $Z\in \bigcap_{i\in\N}[[S_{i}]]$.
If $Z$ is of the form $\sigma^{\frown}0^{\N}$,   put $S'_{i}:=\{\sigma^{\frown}0^{i}\}$. We have $\bar\mu([[S'_{i}]])\le 2^{-i}$ and $Z\in \bigcap_{i\in\N}[[S'_{i}]]$.
Otherwise, put \[S'_{i}:=\{\sigma^{\frown}0^{k}{}^{\frown} {1}: \sigma\in S_{i}, k\in \N\}.\]  Then $T_{S_{i}'}=\{\tau\in 2^{<\N}\mid \tau=\sigma^{\frown}0^{k}$ for some $\sigma\in T_{i}$ and $k\in\N\}$.
Since $T_{S_{i}'}\supseteq T_{S_{i}}$ we have   $\mu(T_{S_{i}'})\le \mu(T_{S_{i}})$. We  still have  $Z\in \bigcap_{i\in\N}[[S_{i}']]$ by the case assumption on  $Z$.
On the other hand, for any $\widehat S\subseteq 2^{<\N}$, if  $ \A x\in 2^{\N}(x\in [[\widehat S]]\to x\in [[S_{i}']]) $   then $T_{\widehat S}\supseteq T_{S_{i}'}$ because $\sigma \in T_{S_{i}'}\to \sigma^{\frown}0^{\N}\in [T_{\widehat S}]$.
Thus, $\bar \mu([[S_{i}']])\le \mu(S_{i}')\le 2^{-i}$.
\end{proof}


%
% ...

\section{Differentiability of functions of bounded variation in $\WWKL$} \label{diffbv}
Lebesgue's theorem states that   functions on $[0,1]$ of bounded variation  are a.e.\ differentiable. The main result of this section, Theorem~\ref{Thm: Leb},  shows that several  versions of  this result are equivalent to $\WWKL$ over $\RCAo$.
 For a function $f$ and distinct reals $a,b $  in the domain of $f$, we denote the slope by  $$S_f(a, b)= \frac{f(b)- f(a)}{b-a}.$$


\begin{definition}[Section 2.3 of \cite{BMN2016}, going back to Demuth, within $\RCAo$] Let $f : \subseteq [0,1] \to \mathbb{R}$ be a function with domain containing $[0,1]_\mathbb{Q}$ ($f$ may be a continuous function or a rationally presented function).
For a given $h>0$,  the $h$-derivative of $f$ at $x\in [0,1]$ is the set of reals defined by
\[ D_{h} f(x)  = \{S_f(a, b) : a,b \in [0,1]_\mathbb{Q} \ \land \ a \le x \le b \ \land \ 0 < b - a < h \}.\]
The function $f$ is \emph{pseudo-differentiable} at $z \in (0,1)$ if 
$\lim_{h \to 0^+} \mathrm{diam}(D_{h} f(z))=0$,
 or more formally, for any $\varepsilon>0$, there exist $c,d\in \Q$ and $h>0$ such that $d-c<\varepsilon$ and $D_{h} f(z)\subseteq[c,d]$.
\end{definition}
The \emph{upper} and \emph{lower pseudo-derivatives of $f$} are defined by 
$\wtilde D f(x) = \lim_{h \to 0^+} \sup D_{h}f(x)$ and
$\D f(x) = \lim_{h \to 0^+} \inf D_{h}f(x)$.
Then, the assertion $\lim_{h \to 0^+} \mathrm{diam}(D_{h} f(z))=0$ formally means that $\D f(z)=\wtilde Df(z)$ without referring to the limits themselves.
The point is that we don't have to require  that $f(z)$ be defined; for instance we could be interested in a function  $f$ only defined on rationals. In this way we can include in our equivalences with $\WWKL$ in Theorem~\ref{Thm: Leb} a statement about functions with rational presentations.  It follows from  \cite[Lemma 2.5]{BMN2016} that if $f$ is defined and  continuous at $z$, then the pseudo-derivative at $z$ exists iff the usual derivative exists, and they agree. 
 
Note that the real $r=\wtilde D f(x)=\D f(x)$ may fail to exist in a model of $\RCAo$ even if $\D f(z)$ and $\wtilde Df(z)$ are equal. We will avoid mentioning the values $\wtilde D f(x)$ or $\D f(x)$ and just consider inequality,  as we already did  in  the case for bounded variation.

\subsection{Pseudo-differentiability of non-decreasing functions within  $\mathsf{RCA_0}$}

\begin{lemma}[a version of part of \text{\cite[Theorem~4.3]{BMN2016} that works within  $\mathsf{RCA_0}$}]\label{lem1}
Let $f:[0,1]_{\mathbb{Q}}\to \mathbb{R}$ be a rationally presented non-decreasing function, and let $z\in [0,1]$ be Martin-L\"of random relative to $f$.
Then  $f$ is pseudo-differentiable at $z$.
\end{lemma}
\begin{proof}
Without loss of generality, we may suppose  that $0\le f(0)\le f(1)\le 1$.
Assume that $f$ is not pseudo-differentiable at $z$.
If $z$ is rational, then $z$ is not Martin-L\"of random, so assume that $z$ is irrational.
We will consider the following two cases.

\smallskip
\noindent
{\it Case 1.} $\D f(z)=\infty$.
Thus, for any $m\in\N$, there exists $a_{0}<z<b_{0}$ such that for any $a,b\in\Q$, $a_{0}<a<z<b<b_{0}\to S_{f}(a,b)>2^{m}$.
For a given $\sigma\in2^{<\N}$, we let $l_{\sigma}=0.\sigma$ and $r_{\sigma}=0.\sigma+2^{-|\sigma|}$.
Let $\varphi(m,\sigma)$ be a $\Pi^{0}_{1}$-formula saying that $S_{f}(l_{\sigma},r_{\sigma})\le 2^{m}$.
Write $\varphi(m,\sigma)\equiv\A s \theta(s,m,\sigma)$ for a $\Sigma^{0}_{0}$-formula $\theta$, and put
\[T_{m}:=\{\sigma\in 2^{<\N}: \A \tau\preceq \sigma \A s<|\sigma|\theta(s,m,\tau)\}.\]
Since $f(1)-f(0)\le 1$, for each $m\in\N$ and $k\in\N$, there are at most $2^{k}$ strings of  length $m+k$ which are not in $T_{m}$.
Thus, $\mu(T_{m})\ge 1-2^{m}$, and hence $\langle 2^{\N}\setminus[T_{m}]: m\in\N \rangle$ forms a Martin-L\"of test.
By the assumption, $z\notin[T_{m}]$ for any $m\in\N$. Thus $z$ is not Martin-L\"of random relative to $f$.

\smallskip
\noindent
{\it Case 2.} $\D f(z)<\infty$ and $\D f(z)<\widetilde{D} f(z)$.
We will follow the proof of (iii) $\to$ (ii) of \cite[Theorem 4.3]{BMN2016} halfway within $\RCAo$.

Since $f$ is non-decreasing and $\D f(z)<\infty$,    the limit $\lim_{y \to z} f(y)$ exists   by   nested interval completeness  \cite[Theorem II.4.8]{SOSOA} which holds in $\RCAo$. So we may assume that $f(z)$ exists and $f$ is continuous at $z$; this 
will be  needed  in the following argument when we apply a  version of  \cite[Lemma 2.5]{BMN2016} formalised within $\RCAo$.

For given $p,q\in \Q$, an interval $A$ is said to be a \emph{$(p,q)$-interval} if it is of the form $A=(pi2^{-n}+q,p(i+1)2^{-n}+q)$ for some $n\in \N$ and $i\in \Z$.
For a finite set $L\subseteq\Q^{2}$, an interval is said to be \emph{$L$-interval} if it is a $(p,q)$-interval for some $(p,q)\in L$.
One can formalise within $\RCAo$ the proofs of Lemma 2.5, Lemma 4.1 and most of the proof of Lemma 4.4 of \cite{BMN2016}.  To see this, note that these arguments only rely on elementary arithmetic, which can be formalised within $\RCAo$.  Hence we have the following.

\begin{claim}\label{claim:algebraic-argument}
There exist rationals $\beta>\gamma>0$ and a finite set $L\subseteq\Q^{2}$ such that
\begin{align*}
 \gamma &> \lim_{h\to 0} \inf \{S_{f}(A): A\mbox{ is an $L$-interval}\wedge |A|\le h\wedge z\in A\},\\
 \beta &< \lim_{h\to 0} \sup \{S_{f}(A): A\mbox{ is an $L$-interval}\wedge |A|\le h\wedge z\in A\}.
\end{align*}
\end{claim}
In the final step of the proof of  \cite[Lemma 4.4]{BMN2016} for both inequalities there, one picks $(p,q)$ and $(r,s)$ from $L$ which by themselves witness   the two  inequalities above, respectively; that is, we only need to look at $(p,q)$ intervals for the first, and at $(r,s)$-intervals for the second.
However, this is impossible within $\RCAo$ since it requires an essential use of the infinite pigeonhole principle (also known as $\mathrm{RT}^{1}$) which is equivalent to $\BII$.
Thus, we need to take  a detour around this part of the proof.

We fix $\beta,\gamma$ and $L\subseteq\Q^{2}$ as in  Claim~\ref{claim:algebraic-argument}.
An \emph{$n$-depth alternation $L$-sequence} is a length $2n+1$ decreasing sequence of $L$-intervals $[0,1]\supseteq A_{0}\supseteq A_{1}\supseteq\dots\supseteq A_{2n}$ such that $S_{f}(A_{2i})<\gamma$ for any $i\le n$, $S_{f}(A_{2i+1})>\beta$ for any $i<n$.
For $(p,q),(r,s)\in L$, an \emph{$n$-depth alternation $(p,q)$;$(r,s)$-sequence} is an $n$-depth alternation $L$-sequence such that $A_{2i}$ is a $(p,q)$-interval for any $i\le n$ and $A_{2i+1}$ is an $(r,s)$-interval for any $i<n$.
The last interval of an $n$-depth alternation $(p,q)$;$(r,s)$-sequence is called an  \emph{$n$-depth $(p,q)$;$(r,s)$-interval}.
By Claim~\ref{claim:algebraic-argument}, for any $n\in\N$, there exists an $n$-depth alternation $L$-sequence such that $z\in A_{2n}$.
Furthermore, by the finite pigeon hole principle, every $n|L|^{2}$-depth alternation $L$-sequence contains an  $n$-depth alternation $(p,q)$;$(r,s)$-subsequence for some $(p,q),(r,s)\in L$.
Thus, we have the following claim.
\begin{claim}\label{claim:n-depth-interval}
For any $n\in\N$ there exist $(p,q),(r,s)\in L$ and an $n$-depth $(p,q)$;$(r,s)$-interval $A$ such that $z\in A$.
\end{claim}
\noindent
Note that we cannot fix $(p,q),(r,s)\in L$ for all $n\in\N$ in this claim since it would  require $\BII$ again.

Next, we will calculate the size of $n$-depth $(p,q)$;$(r,s)$-intervals.
Fix $(p,q),(r,s)\in L$ and let $\{A^{s}\}_{s<u}$ be a finite collection of $n$-depth $(p,q)$;$(r,s)$-intervals.
Take an $n$-depth alternation $(p,q)$;$(r,s)$-sequence $A^{s}_{0}\supseteq\dots\supseteq A^{s}_{2n}=A^{s}$ for each $s<u$, and let $\bar{A}_{i}=\bigcup_{s<u}A^{s}_{i}$.
For a finite union of intervals $\bar{A}$ which is described by a finite disjoint union as $\bar{A}=\bigsqcup_{j<l}[a_{j},b_{j}]$, put $|\bar{A}|:=\sum_{j<l}(b_{j}-a_{j})$ and $\Delta_{f}(\bar{A}):=\sum_{j<l}(f(b_{j})-f(a_{j}))$.
Since any two $(p,q)$-intervals (or $(r,s)$-intervals) are disjoint, or one includes the other, we have that $\Delta_{f}(\bar{A}_{2i})/|\bar{A}_{2i}|<\gamma$ for any $i\le n$ and $\Delta_{f}(\bar{A}_{2i+1})/|\bar{A}_{2i+1}|>\beta$ for any $i<n$.
Thus, for any $i<n$,
\begin{align*}
 |\bar{A}_{2i+2}|\le |\bar{A}_{2i+1}| <\frac{\Delta_{f}(\bar{A}_{2i+1})}{\beta}\le \frac{\Delta_{f}(\bar{A}_{2i})}{\beta}<\frac{\gamma}{\beta}|\bar{A}_{2i}|.
\end{align*}
Hence,
\begin{align}
\bar \mu\left(\bigcup_{s<u} A^{s}\right)=|\bar{A}_{2n}|<\left(\frac{\gamma}{\beta}\right)^{n}|\bar{A}_{0}|\le \left(\frac{\gamma}{\beta}\right)^{n}.\label{eq:size-of-n-depth} 
\end{align}

Now, put
\begin{align*}
U_{n}^{(p,q);(r,s)}&:=\bigcup\{A: A\mbox{ is an $n$-depth $(p,q)$;$(r,s)$-interval}\},\\
U_{n}&:=\bigcup_{(p,q),(r,s)\in L}U_{n}^{(p,q);(r,s)}.
\end{align*}
Note that one can enumerate all $n$-depth $(p,q)$;$(r,s)$-intervals $f$-computably.
By $(\ref{eq:size-of-n-depth})$, $\bar \mu(U_{n}^{(p,q);(r,s)})\le (\gamma/\beta)^{n}$.
Thus, $\bar \mu(U_{n})\le (\gamma/\beta)^{n}|L|^{2}$.
By Claim~\ref{claim:n-depth-interval}, $z\in \bigcap_{n\in\N}U_{n}$. Thus $z$ is not Martin-L\"of random relative to $f$.
\end{proof}
\begin{remark}\label{rem:non-decreasing-differentiability-in-RCA}
By a careful formalization of the notion of   test for computable randomness within $\RCAo$, one can reformulate the above proof to obtain the following assertion within $\RCAo$: {for any rationally presented non-decreasing function $f:[0,1]_{\mathbb{Q}}\to \mathbb{R}$, there exists a computable test relative to $f$ such that $f$ is pseudo-differentiable at $z\in [0,1]$ if $z$ passes the test.}
On the other hand, one can easily see within $\RCAo$ that for any computable test relative to $X$, there exists a real computable from $X$ which can pass it.
Thus, within $\RCAo$,  every rationally presented non-decreasing function is pseudo-differentiable at some point.
However,  as we will see in Theorem~\ref{Thm: Leb}, this   does not  imply that every rational presented non-decreasing function is pseudo-differentiable almost surely (as defined below).  
\end{remark}

\subsection{A.e. pseudo-differentiability of functions of bounded variation} 
We introduce a  notion of a.e.~differentiability in a form that is appropriate within $\RCAo$.
In that setting, any open set $U\subseteq[0,1]$ can be identified with an open set in $2^{\N}$ via the canonical surjection $\pi:2^{\N}\to [0,1]$ defined by $\pi(Z)=\sum_{n\in Z}2^{-n}$.  
We define the measure for open sets in $[0,1]$ by $\bar\mu(U)=\bar\mu(\pi^{-1}(U))$.
This $\bar\mu$ coincides with the Lebesgue measure on $[0,1]$ as  in \cite[p.\ 174]{YS1990}.
\begin{definition}\label{defi:measure-on-real-line} A function $f : \subseteq [0,1] \to \mathbb{R}$ with domain containing $[0,1]_\mathbb{Q}$  is pseudo-differentiable \emph{almost surely} if $\bar \mu(U)= 1$ for any open set $U\subseteq [0,1]$ containing every  point of  pseudo-differentiability of $f$. 
%
\end{definition}

For the  main results of this section we need two preliminaries.
The first one is a model-theoretic generalization of the fact that any Martin-L{\"o}f random real is Martin-L{\"o}f random relative to some $\PA$-degree by Downey, Hirschfeldt, Miller and Nies \cite[Proposition~7.4]{Downey-et-el-2005} and Reimann and Slaman \cite[Lemma~4.5]{Reimann-Slaman-2015}.
\begin{lemma}[Simpson/Yokoyama \cite{SY2011}] \label{lemma2}
For any countable model $(\M, \S) \models \WWKL$ there is $\widehat \S \supseteq \S$ satisfying 
\begin{enumerate}\itemsep0em
\item $(\M, \widehat \S) \models \WKL$, and 
\item for any $A \in \widehat \S$ there is $Z \in \S$ such that $Z$ is Martin-L{\"o}f random relative to $A$.
\end{enumerate}
\end{lemma}

The following is related to a well known result of   Ku$\check{\text{c}}$era; also  see   \cite[Proposition 3.2.24]{nies2009computability}. We say that $W$ is a tail of a set  $Z \subseteq \N$ if there is $n$ such that $W(i) = Z(n+i)$ for each $i$. 
\begin{lem}[$\RCAo$]\label{lem-as}
Let $U\subseteq[0,1]$ be an open set such that $\bar \mu(U)<1$, and let $S\subseteq 2^{<\N}$ be a code for an open set such that $[[S]]=\pi^{-1}(U)$.
Let  $Z\in 2^{\N}$ be  Martin-L\"of random relative to $S$. There exists a tail $W$ of $Z$ such that  $0.W=\pi(Z)\in [0,1]\setminus U$.
\end{lem}
\begin{proof}

Choose  $q\in \Q$ such that $ \mu(S)\le \bar\mu(U) \le q<1$.
Then, $\mu(T_{S})\ge 1-q$.
Let $\widetilde{T} = \{\tau \in 2^{<\mathbb{N}} : \tau \not \in T_{S} \land \tau \restriction_{(|\tau| - 1)} \in T_{S}\}$.
Put 
\[T^{n}:=\{\sigma_{0}^{\frown}\dots^{\frown}\sigma_{k}: k< n\wedge [ \A i< k\, \sigma_{i}\in \widetilde{T}]\wedge \sigma_{k}\in T_{S}\}.\]
Then, we have $\mu(T^{n})\ge 1-q^{n}$.
Thus, for large enough $l\in\N$, $\langle 2^{\N}\setminus [T^{ln}]: n\in\N \rangle$ forms a Martin-L\"of test relative to $S$, and hence $Z\in [T^{ln}]$ for some $n\in\N$.
By $\Sigma^{0}_{1}$-induction, take
\[m=\max\{m': \E c\le ln\, \E \langle \sigma_{i}\in\widetilde{T}: i<c \rangle(Z|_{m'}=\sigma_{0}^{\frown}\dots^{\frown}\sigma_{c-1})\}.\]
Then, the tail $W$ of $Z$ defined by  $W(i)=Z(i+m)$ is in $[T_{S}]$, whence $0.W\in [0,1]\setminus U$.
\end{proof}

\begin{theorem}[$\WWKL$] \label{diffpoint}
Every     rationally presented function of bounded variation is pseudo-differentiable at some point, and  is actually pseudo-differentiable  almost surely.
\end{theorem}

\begin{proof}
We show that the result holds in any countable model $(\M, \S)$ of $\WWKL$. 
Let $f : [0,1]_{\Q} \to \mathbb{R}$ be a   rationally presented function of bounded variation in $(\M, \S)$,
and let $U\subseteq[0,1]$ be an open set such that $\bar\mu(U) < 1$.
We will show that there exists a real $z\in [0,1]\setminus U$ such that $f$ is pseudo-differentiable at $z$.
Let  $(\M, \widehat \S) \models \WKL$ be the model given by  Lemma \ref{lemma2}. By Theorem \ref{jordanq}, $$(\M, \widehat\S) \models \mathsf{Jordan}_\mathbb{Q}.$$ Hence $\widehat \S$ contains a non-decreasing function $g : [0,1]_\mathbb{Q} \to \mathbb{R}$ such that $f \le^*_\mathsf{slope} g$. 

Within $(\M, \widehat \S)$, define $h : [0,1]_\mathbb{Q} \to \mathbb{R}$ by $h(x) = g(x) - f(x)$. By Lemma \ref{lemma2} again, there is a real $z \in [0,1]$ such that $z \in \S$ and $z \in \MLR^{g \oplus h\oplus \mathcal{U}}$.
By Lemma~\ref{lem-as}, we may assume that $z\in [0,1]\setminus U$.
 The functions $g$ and $h$ are pseudo-differentiable at $z$ in $(\M, \widehat \S)$ by Lemma~\ref{lem1}. Therefore $f$ is pseudo-differentiable at $z$ in $(\M, \widehat \S)$, and hence in $(\M, \S)$.
\end{proof}

A continuous function $f:[0,1]\to\R$ is said to be absolutely continuous if for any $\varepsilon>0$ there exists $\delta>0$ such that for any $0\le a_{0}\le \dots \le a_{n}\le 1$ with $a_{n}-a_{0}<\delta$, $\sum_{i<n}|f(a_{i+1})-f(a_{i})|<\varepsilon$.
Note that every absolutely continuous function is of bounded variation within $\RCAo$.
\begin{thm} \label{Thm: Leb}
The following are equivalent over $\RCA$. 
\begin{enumerate}
\item $\WWKLo$
\item Every rationally presented function of bounded variation is pseudo-differentiable   almost surely.
\item Every rationally presented non-decreasing function is pseudo-differentiable   almost surely.
\item Every continuous function of bounded variation is pseudo-differentiable  almost surely.
\item Every effectively uniformly continuous and absolutely continuous function which has a rational presentation is pseudo-differentiable at some point.
\end{enumerate}
\end{thm}
\begin{proof}
1 $\Rightarrow$ 2 is Theorem~\ref{diffpoint}, 2 $\Rightarrow$ 3 is trivial, 2 $\Rightarrow$ 4 is straightforward from Corollary~\ref{cor:rational-presentation-in-RCA}.
For 4 $\Rightarrow$ 5, within $\RCAo$, we have $\bar\mu(\emptyset)=0$,
since if $[[S]]=\emptyset$, then $S$ is empty, so $\mu(S)=0$.
Thus, if an open set $U\subseteq[0,1]$ has  positive measure, then $U$ is not empty.
It remains to show $\neg$1 $\Rightarrow$ $\neg$3 and $\neg$1 $\Rightarrow$ $\neg$5.


We  show $\neg$1 $\Rightarrow$ $\neg$3. Assuming $\neg$1 we    first  construct an open set $U\subseteq[0,1]$ such that $\bar{\mu}(U)<1$ and $[0,1]\setminus U$ only contains rationals. The idea is similar to the one in the proof of Proposition~\ref{prop:ML-random-in-RCA}:
If $\mathrm{WWKL}$ fails,  there is a tree $T$ with no paths  such that $\mu(T)\ge \varepsilon$  where $\varepsilon>0$.
Put 

$T'=\{\tau\in 2^{<\N}\mid \tau=\sigma^{\frown}0^{k}$ for some $\sigma\in T$ and $k<|\tau|\}$, 

$\widetilde T=\{\tau\in 2^{<\N}\mid \tau\notin T'$ and $\tau|_{|\tau|-1}\in T'\}$, and 

$U=\bigcup_{\tau\in \widetilde T}(0.\tau,0.\tau+2^{-|\tau|})\subseteq [0,1]$.

\noindent As in the proof of Proposition~\ref{prop:ML-random-in-RCA}, we have $\bar\mu(2^{\N}\setminus [T'])\le 1-\mu(T)<1$, and any path of $T'$ is rational. Thus $\bar\mu(U)\le \bar\mu(2^{\N}\setminus [T'])<1$ and $[0,1]\setminus U$ only contains rationals.

Next we   construct a rationally presented non-decreasing function which is not pseudo-differentiable at any rational.
Let $\{q_{i}\}_{i\in\N}$ be an enumeration of $[0,1]_{\Q}$.
 Define a function $f:[0,1]_{\Q}\to \R$ by $f(p)=\sum_{q_{i}<p}2^{-i}$.
Clearly, $f$ is non-decreasing and not pseudo-differentiable at any rational.
By Proposition~\ref{prop:rational-presentation-in-RCA},  some vertical shift of $f$ has a rational presentation.
Thus we have $\neg$3.

Finally, we   show $\neg$1 $\Rightarrow$ $\neg$5.
This  implication is related  to    \cite[Theorem~6.7]{BMN2016} (originally due to Demuth) in the setting of reverse mathematics.
If $\mathrm{WWKL}$ fails,  there is a tree $T$ with no path such that $\mu(T)\ge \varepsilon$  where $\varepsilon>0$.
  We   construct a sequence of trees $\langle T_{n}: n\in\N \rangle$ such that no $T_{n}$ has a path and $\mu(T_{n})\ge 1-2^{-4n}$.
  Let the $T^{i}$  be defined  as in the proof of Lemma~\ref{lem-as} where $q = 1- \epsilon$. No $T^{i}$ has a path, and $\mu(T^{i})\ge 1-(1-\varepsilon)^{i}$.
Thus, one can effectively choose $i_{0}<i_{1}<\dots$ so that $\mu(T^{i_{n}})\ge 1-2^{-4n}$. Now let  $T_n = T^{i_n}$. 

Let $\widetilde{T}_{n} = \{\tau \in 2^{<\mathbb{N}} : \tau \not \in T_{n} \ \land \  \tau \restriction_{|\tau| - 1} \in T_{n}\}$. As before   put $I_{\sigma}:=[0.\sigma,0.\sigma+2^{-|\sigma|}]$.
 Since $T_{n}$ has no path we have
 $[0,1]=\bigcup_{\sigma\in \widetilde{T}_{n}}I_{\sigma}$ for any $n$. Since $\mu(T_{n})\ge 1-2^{-4n}$ we have  $\sum_{\sigma\in \widetilde{T}_{n}}|I_{\sigma}|\le 2^{-4n}$.
Note that if $\sigma\in \widetilde{T}_{n}$ and $m<n$,   there exists $\tau\in \widetilde{T}_{m}$ such that $\tau\preceq\sigma$.
Note also that   $|\sigma|\ge n$  for any $\sigma\in \widetilde{T}_{n}$.

For $v \in \mathbb R^+$ and $r \in \mathbb N$, recall $\tup_A(v, r)$ denotes  a sawtooth function on the interval $A$ with $r$ many teeth of height $v$. For each $n,s \in \mathbb{N}$ define a polygonal function $f_{n,s} :[0,1] \to \mathbb{R}$ as follows.
For   $\sigma\in \widetilde{T}_{n}$, on the interval $I_{\sigma}$ , set 
\begin{align*}
f_{n,s} = \begin{cases} \tup_{I_{\sigma}}(2^{-2n-|\sigma|}, 2^{5n})  & \textrm{if $|\sigma|\le s$,} \\ 0 & \textrm{otherwise.} \end{cases} 
\end{align*}
Then $(f_{n,s})_{s\in\mathbb{N}}$ defines an effectively uniformly continuous function $f_{n}$.
For these functions $f_{n}$ one can check the following properties. 
\begin{enumerate}
\renewcommand{\labelenumi}{(\roman{enumi})}
 \item If $\sigma\in \widetilde{T}_{n}$, $x\in I_{\sigma}$ and $m\ge n$, then $0\le f_{m}(x)\le 2^{-2m-|\sigma|}$. In particular, $|f_{n}|\le 2^{-2n}$.\label{enum-nondif-1}
 \item For any $0\le x<y\le 1$ and for any $n\in\N$, $|f_{n}(x)-f_{n}(y)|/|x-y|\le 2^{3n+1}$.\label{enum-nondif-2}
 \item $\v_{f_{n}}\le 2^{-n+1}$.\label{enum-nondif-3}
\end{enumerate}
(i) and (ii) follow from the definition.
To see (iii), 
\[\v_{f_{n}}=\sum_{\sigma\in \widetilde{T}_{n}}\v_{\tup_{I_{\sigma}}(2^{-2n-|\sigma|}, 2^{5n})}=\sum_{\sigma\in \widetilde{T}_{n}}2^{3n-|\sigma|+1}\le 2^{3n+1}2^{-4n}=2^{-n+1}.\]

Define an effectively uniformly continuous function $f$ by $f=\sum_{n\in\N}f_{n}$.
Then, $f$ is of bounded variation since $\v_{f}=\sum_{n\in\N} \v_{f_{n}}\le 2$.
Actually, $f$ is absolutely continuous. One can see this as follows. For any $x\in[0,1]$ and $\varepsilon>0$, take large enough $n\in\N$ so that $\sum_{j>n}\v_{f_{j}}<\varepsilon/2$. Since each   $f_{i}$, $i\le n$,  is absolutely continuous by (ii), one can find $\delta>0$ so that $\sum_{i\le n}(f_{i}(x)-f_{i}(y))<\varepsilon/2$ for any $y$ such that $|x-y|<\delta$.

By Corollary~\ref{cor:rational-presentation-in-RCA}(ii), after  replacing $f$ with a vertical shift we may  assume that $f$ has a rational presentation.

We will see that this $f$ is not pseudo-differentiable at any point.
Let $x\in[0,1]$, $\delta>0$ and $K\in\N$.
We will find $a\le x\le b$ so that $b-a<\delta$ and  $|S_{f}(a,b)|>K$.
Take  $n\in\N$ large enough so that $2^{3n-1}>K$ and $|I_{\sigma}|<\delta$ for any $\sigma\in \widetilde{T}_{n}$.
Since $T_{n}$ has no path, there exists $\sigma\in \widetilde{T}_{n}$ such that $x\in I_{\sigma}$.
Let $a\le x\le b$ so that  $a,b$ are nearest to $x$ yielding extreme values   of the saw-tooth function $\tup_{I_{\sigma}}(2^{-2n-|\sigma|}, 2^{5n})$.
Then, $|f_{n}(b)-f_{n}(a)|=2^{-2n-|\sigma|}$ and $b-a=2^{-5n-1-|\sigma|}$.
By (i), $|f_{j}(b)-f_{j}(a)|\le 2^{-2j-|\sigma|}$ for any $j>n$, and by (ii), $|f_{i}(b)-f_{i}(a)|/|b-a|\le 2^{3i+1}$ for any $i<n$.
Thus,
\begin{align*}
 \frac{|f(b)-f(a)|}{|b-a|}&\ge \frac{|f_{n}(b)-f_{n}(a)|}{|b-a|}-\sum_{j>n}\frac{|f_{j}(b)-f_{j}(a)|}{|b-a|}-\sum_{i<n}\frac{|f_{i}(b)-f_{i}(a)|}{|b-a|}\\
 &\ge\frac{2^{-2n}-|\sigma|}{2^{-5n-1-|\sigma|}}-\sum_{j>n}\frac{2^{-2j-|\sigma|}}{2^{-5n-1-|\sigma|}}-\sum_{i<n}2^{3i+1}\\
 &\ge 2^{3n+1}-2^{3n}-2^{3n-1} = 2^{3n-1}.
\end{align*}
Hence, $|S_{f}(a,b)|>K$. \end{proof}

\begin{remark}
The equivalence of conditions 1 and 3 in the foregoing theorem seems rather strange compared to \cite[Theorem~4.3]{BMN2016} and our discussion in Remark~\ref{rem:non-decreasing-differentiability-in-RCA} since there is no appearance of Martin-L\"of random reals.
Note that the existence of computable random reals doesn't imply $\mathrm{WWKL}$ over $\RCAo$.
This tricky situation may be understood that it is caused by a bad behavior of the Lebesgue measure within $\RCAo$.
For example, one cannot say that \textit{every open set $U\subseteq [0,1]$ is of measure $1$ if $[0,1]\setminus U$ is countable}.
In fact, this is equivalent to $\mathrm{WWKL}$ by the argument in the previous proof.
\end{remark}

\end{document}